\documentclass[11pt,a4paper,reqno]{amsart}
\usepackage{amssymb,amsmath,amsfonts,amsthm,mathtools}
\usepackage{enumerate}
\usepackage[english]{babel}
\usepackage{hyperref}
\usepackage{color}
\newtheorem{theorem}{Theorem}[section]
\newtheorem{lemma}[theorem]{Lemma}
\newtheorem{corollary}[theorem]{Corollary}
\theoremstyle{remark}
\newtheorem{remark}[theorem]{Remark}
\newtheorem{example}[theorem]{Example}
\newcommand{\e}{\mathrm{e}}
\newcommand{\dist}{\operatorname{dist}}
\newcommand{\ran}{\operatorname{Ran}}
\newcommand{\RR}{\mathbb{R}}
\newcommand{\CC}{\mathbb{C}}
\newcommand{\NN}{\mathbb{N}}
\newcommand{\ZZ}{\mathbb{Z}}
\newcommand{\EE}{\mathbb{E}}
\newcommand{\PP}{\mathbb{P}}

\newcommand{\cD}{\mathcal{D}}
\newcommand{\cK}{\mathcal{K}}

\newcommand{\eps}{\varepsilon}
\renewcommand{\phi}{\varphi}
\newcommand{\indic}{\mathbf{1}}
\newcommand{\lr}{\left(}
\newcommand{\rr}{\right)}
\DeclareMathOperator{\Ran}{Ran}
\DeclareMathOperator{\Tr}{Tr}
\DeclareMathOperator*{\supp}{supp}
\newcommand*\Diff[1]{\mathop{}\!\mathrm{d}#1}
\newcommand{\norm}[1]{\ensuremath{\left\lVert #1 \right\rVert}}
\newcommand{\bes}{\begin{equation*}}
\newcommand{\ees}{\end{equation*}}
\newcommand{\be}{\begin{equation}}
\newcommand{\ee}{\end{equation}}
\newcommand{\eqs}[1]{\begin{align*}#1\end{align*}}
\newcommand{\eq}[1]{\begin{align}#1\end{align}}
\newcommand{\abs}[1]{\left\lvert #1 \right\rvert}
\newcommand{\grad}{\nabla}
\renewcommand{\L}{\Lambda}
\renewcommand{\restriction}{|}
\newcommand{\per}{\mathrm{per}}
\definecolor{darkred}{rgb}{0.5,0,0}
\definecolor{darkgreen}{rgb}{0,0.5,0}
\definecolor{darkblue}{rgb}{0,0,0.5}
\hypersetup{
	colorlinks,
	linkcolor=darkblue,
	filecolor=darkgreen,
	urlcolor=darkred,
	citecolor=darkblue,
	pdftitle={Quantitative unique continuation for spectral subspaces of Schr{\"o}dinger operators with singular potentials},
	pdfauthor={Alexander Dicke, Christian Rose, Albrecht Seelmann, Martin Tautenhahn},
}
\title[Unique continuation with singular potentials]{Quantitative unique continuation for spectral subspaces of Schr{\"o}dinger operators with singular potentials}

\author[A.~Dicke]{Alexander Dicke}
\address{
	A.~Dicke,
	Technische Univer\-si\-t\"at Dortmund, Fakult\"at f\"ur Mathematik,
	44227 Dortmund, Germany
}
\email{adicke.math@gmail.com}
\author[C.~Rose]{Christian Rose}
\address{
	C.~Rose,
	Universität Potsdam, Institut für Mathematik, 14476 Potsdam, Germany
}
\email{rosecmath@googlemail.com}
\author[A.~Seelmann]{Albrecht Seelmann}
\address{
	A.~Seelmann,
	Technische Univer\-si\-t\"at Dortmund, Fakult\"at f\"ur Mathematik,
	44227 Dortmund, Germany
}
\email{albrecht.seelmann@mathematik.tu-dortmund.de}
\author[M.~Tautenhahn]{Martin Tautenhahn}
\address{
	M.~Tautenhahn,
	Univer\-si\-t\"at Leipzig, Fakult{\"a}t f{\"u}r Mathematik und Informatik, 04109 Leipzig, Germany
}
\email{martin.tautenhahn@math.uni-leipzig.de}
\begin{document}
%
%
\begin{abstract}
	Recent (scale-free) quantitative unique continuation estimates for spectral subspaces of Schr{\"o}dinger operators are extended 
	to allow singular potentials such as certain $L^p$-functions. 
	The proof is based on accordingly adapted Carleman estimates.
	Applications include Wegner and initial length scale estimates for random Schr{\"o}dinger operators and control 
	theory for the controlled heat equation with singular heat generation term.
\end{abstract}

\maketitle
%
%
\section{Introduction} 
The focus of the present note is laid on~\emph{quantitative unique continuation estimates}. 
Such estimates are inequalities of the type 
\begin{equation} \label{eq:intro}
	\lVert \psi \rVert_{L^2 (S)}^2 \geq C \lVert \psi \rVert^2_{L^2 (\Lambda)}	
\end{equation}
for functions $\psi$ in the spectral subspace of a lower semibounded self-adjoint operator $H$ in $L^2 (\Lambda)$ up to energy $E$,
where $S \subset \Lambda \subset \RR^d$ and $C$ is a positive constant depending on the geometry of the set $S$, the energy $E$, and the operator $H$. 
Of special interest is an explicit form of this dependence.
In particular, inequality~\eqref{eq:intro} is called scale-free if the constant $C$ is independent of $\Lambda$.
If $H$ is the negative Laplacian, $\Lambda = \RR^d$ and $S \subset \RR^d$ is a so-called thick set, then inequality~\eqref{eq:intro} follows from the
well-known Logvinenko-Sereda theorem~\cite{LogvinenkoS-74}, see also~\cite{Panejah-61, Panejah-62, Kacnelson-73, Kovrijkine-thesis, Kovrijkine-01} and \cite{EgidiV-20,Egidi-21,EgidiS-21} for the case of different domains $\Lambda$.
Nowadays, there is a huge amount of literature on quantitative unique continuation estimates under various assumptions on $\Lambda$, the operator $H$, geometric properties 
of the set $S$, or the length on the energy interval, see, e.g.,~\cite{LebeauR-95,LebeauZ-98,JerisonL-99,RojasMolinaV-13,Klein-13,NakicTTV-18,NakicTTV-18-JST,StollmannS-21,LebeauM-arxiv}. 
Quantitative unique continuation estimates have various applications in several fields of mathematics such as control theory for the heat equation, 
see, e.g.,~\cite{LebeauR-95,LebeauZ-98,JerisonL-99,NakicTTV-20,EgidiNSTTV-20}, as well as the theory of random Schr{\"o}dinger
operators, see, e.g.,~\cite{BourgainK-05,CombesHK-07,GerminetK-13,BourgainK-13,RojasMolinaV-13,Klein-13,TaeuferT-18,NakicTTV-18,MuellerRM-22,SeelmannT-20}. 

In case of Schr{\"o}dinger operators $H = -\Delta + V$, scale-free quantitative unique continuation has almost exclusively been considered for~\emph{bounded} potentials $V$. 
There has been a recent attempt to remove this boundedness assumption in~\cite{KleinT-16a}, but its result is restricted to short energy intervals only. 
As a consequence, it is not suitable for certain applications, e.g., in control theory or in the theory of random Schr{\"o}dinger operators.
Other results dealing with singular lower order terms, such as~\cite{SchechterS-80,Wolff-92a,KochT-01}, are not quantitative and concern only qualitative unique continuation.

In the present note we extend earlier results on quantitative unique continuation for bounded potentials~\cite{NakicTTV-18-JST,NakicTTV-18} to a certain class of unbounded potentials. Compared to the above mentioned~\cite{KleinT-16a}, our result holds for energy intervals of arbitrary length, and our class of unbounded potentials covers and extends (for $d\geq 2$) the ones from~\cite{KleinT-16a}. 

The paper is organized as follows: The basic notations and the statement of the main result, Theorem~\ref{thm:ucp}, are given in Section~\ref{sec:results}, followed by a discussion of several applications in Section~\ref{sec:application}.
In Section~\ref{sec:potentials} we provide some technical results and examples for the class of potentials we consider. 
Thereafter, Section~\ref{sec:Carleman} is devoted to adaptations of two Carleman estimates, and in Section~\ref{sec:UCP} we conclude the proof of 
of the main result, where the previously mentioned Carleman estimates play an essential role. 
Some technical details are postponed to Appendix~\ref{sec:ghost-dimension}.
%
%
\section{The main result} \label{sec:results}
Let $\Lambda$ be a~\emph{generalized rectangle} of the form $\Lambda = \bigtimes_{j=1}^d (a_j,b_j)$ for some $a_j,b_j\in\RR\cup\{\pm\infty\}$, $a_j<b_j$, $j=1,\dots,n$,
and let $V\colon\L\to\RR$ be measurable such that the domain of the associated self-adjoint multiplication operator in $L^2(\L)$ contains $H^1(\L)$, that is, $\cD(V)\supset H^1(\L)$. 
We call such potentials $V$~\emph{admissible} (on $\L$). 

We show in Lemma~\ref{lem:operator-form-bnd} below that an admissible potential on $\L$ is infinitesimally $\Delta$-bounded on $L^2(\L)$ and satisfies 
\be
	\norm{V\psi}_{L^2(\L)}^2 \leq \lambda_1 \norm{\grad\psi}_{L^2(\L)}^2+\lambda_2\norm{\psi}_{L^2(\L)}^2,\quad\psi \in H^1(\L),
	\label{eq:form-bounded}
\ee
for some $\lambda_1,\lambda_2\geq 0$.
Thus, the Schr{\"o}dinger operator
\[
	H_\L^\bullet=-\Delta_\L^\bullet+V\colon L^2(\L)\supset \cD(\Delta_\L^\bullet)\to L^2(\L),\quad \bullet\in\{D,N,\per\},
\]
is self-adjoint and lower semibounded, where $-\Delta_\L^D$ and $-\Delta_\L^N$ denote the Dirichlet and Neumann realizations of the Laplacian on $\L$, 
and $-\Delta_\L^\per$ denotes the Laplacian with periodic boundary conditions (where applicable), respectively. 
 
Let $G>0$ and $\delta\in (0,G/2)$. 
We consider~\emph{equidistributed sets} of the form
\[
	S_\delta =\bigcup_{j} B_j,
\]
where the union is taken over all $j\in(G\ZZ)^d$ with $j+(-G/2,G/2)^d\subset\L$ and each $B_j$ is a ball of radius $\delta$ contained in $j+(-G/2,G/2)^d$;
we always assume that $\L$ is such that the above union for $S_\delta$ is non-empty.
Let us emphasize that unique continuation estimates from equidistributed sets were considered earlier in \cite{RojasMolinaV-13,Klein-13}. 
\par
Our main result reads as follows.

\begin{theorem} \label{thm:ucp}
	Let $H_\L^\bullet$ and $S_\delta\neq\emptyset$ be as above. 
	Then there is a constant $N>0$, depending only on the dimension $d$, such that for all
	$E\in\RR$ and all $\psi\in \Ran P_{H_\L^\bullet}(E)$ we have
	\[
		\norm{\psi}_{L^2(S_\delta)}^2
		\geq
		(\delta/G)^{N\cdot(1+G^2\lambda_1+G^{4/3}\lambda_2^{1/3}+G\sqrt{\max\{0,E\}})}\norm{\psi}_{L^2(\L)}^2,
	\]
	where $P_{H_\L^\bullet}(E)$ denotes the spectral projection of $H_\L^\bullet$ associated with the interval $(-\infty,E]$ and 
	$\lambda_1,\lambda_2\geq 0$ are as in~\eqref{eq:form-bounded}.
\end{theorem}

The theorem covers the main result of~\cite{NakicTTV-18-JST} for essentially bounded potentials $V$ (with $\lambda_1 = 0$ and
$\lambda_2 = \norm{V}_\infty^2$) and extends it to the class of admissible potentials, which contains certain singular and, in
particular, unbounded potentials that were not treated before, see Example \ref{ex:admissible-potentials} below and the discussion thereafter. 

It should also be mentioned that the statement of Theorem~\ref{thm:ucp} is not necessarily limited to the case where in each coordinate the same boundary conditions are imposed. 
In principle, we may instead impose in each coordinate of $\Lambda$ Dirichlet, Neumann, or (quasi-)periodic 
(if applicable) boundary conditions separately with essentially the same proof. 
One only has to make sure that the domain of the corresponding Laplacian on $\Lambda$ belongs to $H^2(\Lambda)$, which is guaranteed, for instance, by~\cite[Example~4.2]{Seelmann-21}.

\section{Applications} \label{sec:application}
Let $V\colon\RR^d\to\RR$ be admissible on $\RR^d$. 
Then the constants from~\eqref{eq:form-bounded} can be chosen independently from $\L$, provided that $\L$ contains a cube of a given, 
fixed, side length, see Lemma~\ref{lem:extension} below.
In this sense, the quantitative unique continuation estimate from Theorem~\ref{thm:ucp} is~\emph{scale-free}. 
This property carries over into the constants appearing in our applications.

\subsection{Control theory} 
We consider the controlled heat equation with a singular heat generation term $-V$, that is,
\be \label{eq:controlled-heat-equation}
	\dot u(t) + H_\Lambda^\bullet u (t) =\indic_{S_\delta} f (t), \quad
	t \in (0,T], \quad
	u (0) = u_0 \in L^2 (\Lambda),
\ee
where $f \in L^2 ((0,T) ; L^2 (\Lambda))$ is a so-called control function and $T > 0$ is a given final time. 
The system~\eqref{eq:controlled-heat-equation} is called~\emph{null-controllable in time $T > 0$} if for all initial states $u_0 \in L^2 (\Lambda)$ there exists 
a control function $f \in L^2 ((0,T) ; L^2 (\Lambda))$ such that the mild solution to~\eqref{eq:controlled-heat-equation} satisfies
\[
	u (T) = \mathrm{e}^{- H_\Lambda^\bullet T} u_0 + \int_0^T \mathrm{e}^{- H_\Lambda^\bullet (T-s)} \indic_{S_\delta} f (s) \mathrm{d} s = 0 .
\]
Moreover, the \emph{control cost} in time $T$ is defined as
\[
	C_T = \sup_{\lVert u_0 \rVert = 1} \inf \bigl\{ \lVert f \rVert_{L^2 ((0,T) ; L^2 (\Lambda))} \colon u (T) = 0 \bigr\} .
\]

An immediate application of Theorem~\ref{thm:ucp} is that the controlled heat equation~\eqref{eq:controlled-heat-equation} is null-controllable in time $T$ and the cost satisfies
\be \label{eq:cost}
	C_T \leq \frac{1}{\sqrt{T}} \left(\frac{G}{\delta}\right)^{K (1 + G^2 \lambda_1 + G^{4/3} \lambda_2^{1/3})}
	\exp \left( \frac{K G^2 \ln^2 (G / \delta)}{T}  - \kappa_- T \right) ,
\ee
where $\kappa_- = \min \{ \kappa , 0 \}$ with $\kappa = \inf \sigma (H_\Lambda^\bullet)$, and 
$K>0$ is a constant depending only on the dimension $d$.

We briefly explain the main steps needed for the proof of this statement.
The first step is a reformulation of Theorem~\ref{thm:ucp}. More precisely, 
for all $\psi \in \ran P_{H_\Lambda^\bullet} (E)$ and all $E \geq \kappa$ we have
\begin{equation*}
	\lVert \psi \rVert_{L^2 (\Lambda)}^2 \leq d_0 \mathrm{e}^{d_1 \sqrt{E - \kappa_-}} \lVert \psi \rVert_{L^2 (S_\delta)}^2 ,
\end{equation*}
where
\[ 
	d_0 = (G / \delta)^{N (1 + G^2 \lambda_1 + G^{4/3} \lambda_2^{1/3})}\quad\text{and}\quad d_1 = N G \ln (G / \delta).
\]
It is well known that inequalities of this type imply a so-called \emph{final state observability estimate} 
for the dual system to~\eqref{eq:controlled-heat-equation} with an (explicitly given) observability constant $C_{\mathrm{obs}}$, 
see, e.g.,~\cite{LebeauR-95,FursikovI-96,Miller-10,TenenbaumT-11,WangZ-17,BeauchardP-18,NakicTTV-20}. 
By Douglas' lemma~\cite{Douglas-66}, see also~\cite{DoleckiR-77}, this observability estimate implies that the controlled heat 
equation~\eqref{eq:controlled-heat-equation} is null controllable in time $T$, and that the control cost in time $T$ satisfies $C_T \leq \sqrt{C_{\mathrm{obs}}}$. 
The desired estimate on the control cost~\eqref{eq:cost} follows in this framework from Theorem~2.8 (if $\kappa \geq 0$) and Theorem~2.12 (if $\kappa < 0$) in~\cite{NakicTTV-20}. 

\subsection{Random Schr{\"o}dinger operators}

In what follows, we denote by $\L_l(y)$ the cube of side length $l>0$ centered at $y\in\RR^d$.
Let $V_0$ be an admissible potential, $H_0 =-\Delta+V_0$, and let $H_\omega =H_0+V_\omega$ with $V_\omega=\sum_{j\in\ZZ^d}\omega_ju_j$,
where $(\omega_j)$ are independent and uniformly distributed random variables on $[0,1]$ and where $u_j\in L^p(\RR^d)$, $p>d/2$, is supported in $\L_1(j)$.
Moreover, we assume  that there are balls $B_j \subset \Lambda_1 (j)$ of radius $\delta>0$ such that $\indic_{B_j}\leq u_j$ for all $j\in\ZZ^d$.
We set $W =\sum_{j\in\ZZ^d}u_j$. 
For this type of random Schr{\"o}dinger operator, under suitable additional assumptions, our main result now implies~\emph{initial length scale} and~\emph{Wegner estimates}.
Such estimates are useful, for instance, in proofs of localization via multi-scale analysis, see, e.g.,~\cite{Stollmann-01,GerminetK-01,GerminetK-13} and the references therein.

Note that the statements below hold true also for more general models, but for brevity and simplicity, we only consider the model introduced here. 

\subsubsection{Initial length scale estimate} 
We assume that there exists an interval $(a,b)\subset\RR$ that, for each $L\in\NN$, $x\in\ZZ^d$ and $t\in [0,1]$, belongs to the resolvent set of the box restriction 
$H^\bullet_{0,\L_L(x)} + t  W \restriction_{\L_L(x)}$; cf.~condition (H3') in~\cite{SeelmannT-20}. 
Then, upon replacing Proposition~3.1 in~\cite{SeelmannT-20} by Theorem~\ref{thm:ucp}, in dimension $d\geq 3$, the reasoning in the last mentioned paper can easily be
adapted to show that for all $q>0$ and all $\alpha\in (0,1)$ there is some $L_0\in\NN$ such that for all $L\geq L_0$ and all $x\in\ZZ^d$ the random operator $H_\omega$ satisfies 
\[
	\PP(\sigma(H^\bullet_{\omega,\L_L(x)})\cap [b,b+L^{-\alpha})=\emptyset)\geq 1-L^{-q}.
\]

The proof relies on a delicate interplay of certain parameters. 
Roughly speaking, one needs to trade for large $G$ a lower bound for the probability of the event that at least one site $j$ in each cube of a family of cubes of 
side length $G$ satisfies $\omega_j\geq 1/2$, which is of the form $1-G^{-d}$, against the constant of the unique continuation estimate in terms of $G$.

With this in mind, in the current setting of Theorem~\ref{thm:ucp} the lower bound for the quotient $\norm{\psi}_{L^2(S_\delta)}^2 / \norm{\psi}_{L^2(\L)}^2$ 
is proportional to $\exp(-G^{2+\eps})$ for large $G$, as compared to $\exp(-G^{4/3+\eps})$ for bounded $V$ in~\cite{SeelmannT-20}.
As a consequence, the corresponding reasoning is now restricted to the case $d \geq 3 > 2 + \eps$.

It is however possible to provide a subclass of admissible potentials for which a more favorable dependence of the constant on
the scaling parameter $G$ is available and, thus, a lower bound proportional to $\exp(-G^{s+\eps})$ for large $G$ and some
$s < 2$; the latter allows also $d = 2$ in the above mentioned reasoning, as considered in~\cite{SeelmannT-20}. 
To be more precise, a subclass for which this holds is described as follows:
Suppose that $V\colon \RR^d\to\RR$ is measurable such that $V^2$ is admissible. 
Then, the considerations in Section~\ref{sec:potentials} below show that also $V$ is admissible and that there are $a,b \geq 0$ such that on every generalized 
rectangle $\Lambda$ with $S_\delta\neq\emptyset$ the constants $\lambda_1,\lambda_2$ in~\eqref{eq:form-bounded} can be chosen as $\lambda_1 = a/G$ and $\lambda_2 = Ga + b$. 
For large $G$, this leads to $G^2\lambda_1 = Ga$ and $G^{4/3}\lambda_2^{1/3} \leq G^{5/3}(a+b)^{1/3}$; hence we obtain a lower bound for
$\norm{\psi}_{L^2(S_\delta)}^2 / \norm{\psi}_{L^2(\L)}^2$ proportional to $\exp(-G^{5/3+\eps})$.

\subsubsection{Wegner estimate} 
We assume that the single-site potentials $(u_j)$ are uniformly admissible. 
By this we mean that the corresponding constants from~\eqref{eq:form-bounded} can be chosen independently of $j$.
Hence, $W$ is admissible and so is the sum $V_0+W$.
Note that $0\leq V_\omega\leq W$ implies that $V_0+V_\omega$ is also admissible.
In view of Lemma~\ref{lem:extension}\,(b) below, the corresponding constants from~\eqref{eq:form-bounded} on $\L_L$, $L\in\NN$,
can then be chosen independently of $\omega$ and $L$; we denote them by $\lambda_1$ and $\lambda_2$.

Our main result implies that for every $E_0\in\RR$ there are constants $C,\eps_0>0$ and $\tau\in (0,1)$ such that for all $\eps\in (0,\eps_0)$, all $L\in\NN$, and all $E\in\RR$ 
satisfying $E+3\eps\leq E_0$ we have
\[
	\EE\bigl (\Tr \chi_{[E-\eps,E+\eps]}(H_{\omega,\L_L}^\bullet) \bigr)\leq C\eps^\tau L^d.
\]
This is obtained by following the proof in~\cite{NakicTTV-18}. Denoting the eigenvalues of a non-negative operator $H$ with
discrete spectrum (enumerated non-decreasingly and counting multiplicities) by $E_k(H)$, $k \in \NN$, one only needs to verify
the corresponding eigenvalue lifting for admissible potentials, i.e., the fact that
\[
	E_k(H_{\omega,\L_L}^\bullet+\eps W\restriction_{\L_L})\geq E_k(H_{\omega,\L_L}^\bullet)+\eps^{1/\tau},\quad\eps\leq\eps_0,
\]
for some $\eps_0>0$, $\tau\in (0,1)$ depending only on $E_0$ and the model parameters. 
This estimate follows easily from our main result as soon as we have verified that for all $\omega$ with $E_k(H_{\omega,\L_L}^\bullet)\leq E_0$ 
we have $E_k(H_{\omega,\L_L}^\bullet+\eps W\restriction_{\L_L})\leq E_0+K$ for some constant $K\geq 0$ depending only on $E_0$, $\lambda_1$, and $\lambda_2$, 
see the proof of Theorem~2.10 in \cite{NakicTTV-18} for details.

In order to verify such an upper bound, we let $M_k\subset\Ran P_{H_{\omega,\L_L}^\bullet}(E_0)$ be the linear span of the eigenfunctions of $H_{\omega,\L_L}^\bullet$
associated to the first $k$ eigenvalues, enumerated non-decreasingly and counting multiplicities.
Then the variational characterization of eigenvalues implies
\[
	E_k(H_{\omega,\L_L}^\bullet+\eps W\restriction_{\L_L}) 
	\leq\sup_{\substack{\psi\in M_k\\\norm{\psi}_2=1}}\left[\langle\psi,H_{\omega,\L_L}^\bullet\psi \rangle_{L^2(\L_L)}
		+\langle\psi,\eps W\restriction_{\L_L}\psi\rangle_{L^2(\L_L)}\right].
\]
Now the arguments used in the proof of Lemma~\ref{lem:norm-compare} below show that there is a constant $K=K(E_0,\lambda_1,\lambda_2)$ such that for all 
normalized $\psi\in M_k$ it holds $\langle\psi,H_{\omega,\L_L}^\bullet\psi \rangle_{L^2(\L_L)}\leq E_0$ and $\langle\psi, W\restriction_{\L_L}\psi\rangle_{L^2(\L_L)} \leq K$.
Hence, we obtain the desired inequality for $E_k(H_{\omega,\L_L}^\bullet+\eps W\restriction_{\L_L})$.

It should be noted that Wegner estimates with unbounded single-site potentials already exist for more restrictive models, 
cf., e.g.,~\cite{KirschSS-98a,KirschSS-98b,CombesHN-01} or the survey~\cite{Veselic-08}. 
%
%
\section{The class of admissible potentials} \label{sec:potentials}
First we provide some properties of admissible potentials that are the core of our considerations.

\begin{lemma} \label{lem:operator-form-bnd}
	Let $V \colon \L\to\RR$ be admissible. 
	Then
	\begin{enumerate}[(a)]
		\item
		$V$ is infinitesimally operator bounded with respect to the Laplacian $\Delta_\L^\bullet$ on $L^2(\L)$, $\bullet\in\{D,N,\per\}$.
		More precisely, there are constants $a,b \geq 0$ such that
		\be
			\norm{V\psi}_{L^2(\L)}
			\leq
			a\eps \norm{\Delta_\L^\bullet\psi}_{L^2(\L)} + \Bigl( \frac{a}{\eps} + b \Bigr)\norm{\psi}_{L^2(\L)}
			\label{eq:op-bnd}
		\ee
		for all $\psi\in\cD(\Delta^{\bullet}_\L)$ and all $\eps>0$.
		\item There are constants $\lambda_1,\lambda_2\geq 0$ such that 
		\be
			\norm{V\psi}_{L^2(\L)}^2 \leq \lambda_1 \norm{\grad\psi}_{L^2(\L)}^2+\lambda_2\norm{\psi}_{L^2(\L)}^2
			\label{eq:form-bnd}
		\ee
		for all $\psi\in H^1(\L)$.
	\end{enumerate}
\end{lemma}

\begin{proof}[Proof of Lemma~\ref{lem:operator-form-bnd}]
	(a)
	We clearly have $\cD((-\Delta_\L^\bullet)^{1/2}) \subset H^1(\Lambda) \subset \cD(V)$. 
	Hence, the claim follows from the more general statement in~\cite[Corollary 2.1.20]{Tretter-08} and its proof.
	
	(b)
	We have $\cD(\nabla) = H^1(\L) \subset \cD(V)$, where $\nabla$ denotes the gradient as a closed operator in $H^1(\L)$. 
	This implies that $V$ is relatively bounded with respect to the gradient, see, e.g.,~\cite[Remark~IV.1.5]{Kato-80}, which agrees with the claim.
\end{proof}

Note that Lemma~\ref{lem:operator-form-bnd}\,(a) ensures that the operator sum $H_\L^\bullet = -\Delta_\L^\bullet + V$ with an admissible 
potential $V$ is self-adjoint on $\cD(\Delta_\L^\bullet)$ and lower semibounded.

\begin{remark} \label{rem:admissible-potentials}
	(1)
	$V$ is admissible if $V^2$ is form bounded with respect to the Neumann Laplacian $-\Delta_\L^N$. 
	Indeed, an inequality like~\eqref{eq:form-bnd} then holds for all $\psi \in \cD(\Delta_L^N) \subset \cD(V^2)$, 
	which extends to \eqref{eq:form-bnd} by taking the closure for $\grad$.

	(2)
	If $V \colon \L \to \RR$ is measurable with $\cD(\Lambda_\L^N) \subset \cD(V^2)$, then $V$ is admissible.
	Indeed, in this case, $V^2$ is $\Delta_\L^N$-bounded, see, e.g.,~\cite[Remark~IV.1.5]{Kato-80}. 
	The claim follows from the fact that $V^2$ is also form bounded with respect to $-\Delta_\L^N$, see, e.g.,~\cite[Theorem~VI.1.38]{Kato-80}, and part (1).
\end{remark} 

As a consequence of the above considerations, we obtain the following result, which describes a subclass of admissible potentials
that allows for a more explicit control of the involved constants $\lambda_1$ and $\lambda_2$.

\begin{corollary}
	Let $V \colon \Lambda \to \RR$ be measurable and assume that $V^2$ is admissible. 
	Then $V$ is admissible, and there are $a,b \geq 0$ such that for every $\eps > 0$ the corresponding constants $\lambda_1$, $\lambda_2$ from
	Lemma~\ref{lem:operator-form-bnd}\,(b) can be chosen as $\lambda_1 = a\eps$ and $\lambda_2 = a/\eps + b$.
\end{corollary}

\begin{proof}
	We already know from Lemma~\ref{lem:operator-form-bnd}\,(a) that $V^2$ is infinitesimally operator bounded with respect to $\Delta_\Lambda^N$. 
	Let $\tilde{a},\tilde{b} \geq 0$ denote the corresponding constants in~\eqref{eq:op-bnd}.
	Following the proof of~\cite[Theorem~VI.1.38]{Kato-80}, we obtain that $V^2$ is form bounded with respect to
	$-\Delta_\Lambda^N$ with
	\[
		\norm{V\psi}_{L^2(\L)}^2
		\le
		2\tilde{a}\eps \langle (-\Delta_\L^N)\psi , \psi \rangle_{L^2(\L)} + \Bigl( \frac{2\tilde{a}}{\eps} + 2\tilde{b} \Bigr)
		\norm{\psi}_{L^2(\L)}^2
	\]
	for all $\psi \in \cD(\Delta_\L^N) \subset \cD(V^2)$ and all $\eps > 0$. 
	The claim with $a = 2\tilde{a}$ and $b = 2\tilde{b}$ now follows from Remark~\ref{rem:admissible-potentials}\,(1).
\end{proof}%

We now give some examples of admissible potentials. 

\begin{example}\label{ex:admissible-potentials}
	
	(1) Every real-valued $V \in L^\infty(\L)$ is admissible.

	(2) It follows from~\cite[Lemma~2.1]{KirschM-82c} that $V^2$ is form bounded with respect to $-\Delta^N_\L$ if $V\in L^p(\L)$ with
		$p\geq d$ for $d\geq 3$, $p>2$ for $d=2$ and $p=2$ for $d=1$.  
		Therefore, by Remark \ref{rem:admissible-potentials}\,(1), such potentials are admissible. 
	
	(3) If $V \colon \RR^d \to \RR$ is measurable such that $V^2$ belongs to the Kato class in $\RR^d$ (see, e.g.,~\cite{AizenmanS-82} or~\cite[Section~1.2]{CyconFKS-87} 
		for a discussion), then $V^2$ is infinitesimally form bounded with respect to the Laplacian.
		Hence $V$ is admissible on $\RR^d$ by Remark \ref{rem:admissible-potentials}\,(1). 
		In particular, this is the case if $V$ belongs to $L_{\text{loc},\text{unif}}^p(\RR^d)$ with $p > d$ for $d \geq 2$ and $p = 4$ for $d = 1$.
\end{example}

Since the admissible potentials clearly form a vector space, also sums of potentials from (1) and (2) (resp.~(1) and (3)) in Example \ref{ex:admissible-potentials} are admissible. 
Therefore, in dimension $d\geq 2$, this class essentially covers the singular potentials considered in~\cite{KleinT-16a}. 
However, potentials of the type (3) in the previous example have not been discussed in the latter paper. 

Obviously, we can always trivially extend an admissible potential on some generalized rectangle to the whole 
of $\RR^d$, and the corresponding constants $\lambda_1$ and $\lambda_2$ from Lemma~\ref{lem:operator-form-bnd}\,(b) 
carry over. 
However, for the considerations in Section~\ref{sec:UCP} below, this way of extension
does not retain enough information of the potential on $\RR^d \setminus \L$. 
We first need to extend it in a more sophisticated way to a large enough generalized rectangle containing $\L$. 
More precisely, if $\L = \bigtimes_{j=1}^d (a_j,b_j) \neq \RR^d$, we extend $V$ to a potential $\tilde{V} \colon \tilde{\L} \to \RR$ on
\[
	\tilde{\L} = \bigtimes_{j=1}^d (\tilde{a}_j , \tilde{b}_j)
	\quad\text{ with }\quad
	(\tilde{a}_j , \tilde{b}_j) = (a_j - (b_j-a_j) , b_j + (b_j-a_j))
\]
by symmetric reflections with respect to the boundary hyperplanes of $\L$ in case of Dirichlet or Neumann boundary conditions and
periodically in case of (quasi-)periodic boundary conditions; if we work with mixed boundary conditions, this extension has to be
done in each coordinate separately.

The next lemma shows that the extended potential is admissible (on $\tilde{\L}$) with the same constants.
This ensures that the procedure may be repeated as many times as needed before extending trivially to the whole of $\RR^d$.

\begin{lemma}\label{lem:extension}
	\begin{enumerate}[(a)]
	\item Let $V \colon \L \to \RR$ be admissible, and define $\tilde{\L}$ and $\tilde{V} \colon \tilde{\L} \to \RR$ as above. 
		Then, $\tilde{V}$ is admissible on $\tilde{\L}$, and the corresponding constants $\lambda_1,\lambda_2$ 
		from Lemma~\ref{lem:operator-form-bnd}\,(b) agree with those of $V$ on $\L$.
	
	\item Let $V \colon \RR^d\to \RR$ be admissible on $\RR^d$ and $\L$ be a generalized rectangle containing a cube of side length $L_0>0$. 
		Then $V\restriction_\L$ is admissible on $\L$, and the corresponding constants from Lemma~\ref{lem:operator-form-bnd}\,(b)
		can be chosen such that they depend only on $L_0$, $d$, and the constants $\lambda_1,\lambda_2$ associated to $V$ on $\RR^d$.
	\end{enumerate} 
\end{lemma} 

\begin{proof}
	(a) Let $\psi \in H^1(\tilde{\L})$. 
	By construction of $\tilde{\L}$ and $\tilde{V}$, there are affine transformations $R_1, \dots, R_l \colon \RR^d \to \RR^d$, $l \in \NN$, each $R_j$ being either a reflection	with respect to a boundary
	hyperplane of $\L$ or a translation, such that $L^2(\tilde{\L}) = L^2(\L) \oplus \bigoplus_{j=1}^l L^2(R_j\L)$ and $\tilde{V}\circ R_j \restriction_\L = V$. 
	For each $j \in \{ 1, \dots, l\}$ we then have $\psi \circ R_j \restriction_\L \in H^1(\L)$ and, therefore,
	\eqs{
		\lVert \tilde{V} \psi \rVert_{L^2(R_j\L)}^2
		=\norm{V (\psi \circ R_j)}_{L^2(\L)}^2
		&\leq \lambda_1\norm{(\grad\psi)\circ R_j}_{L^2(\L)}^2+\lambda_2\norm{\psi\circ R_j}_{L^2(\L)}^2\\
		&= \lambda_1\norm{\grad\psi}_{L^2(R_j\L)}^2+\lambda_2\norm{\psi}_{L^2(R_j\L)}^2.
	}
	Summing over $j$ concludes the proof. 

	(b) Let $\psi \in H^1(\L)$ and let $\tilde{\L}$ be as above. 
	We extend $\psi$ to $\tilde{\L}$ by symmetric reflections with respect to the boundary surfaces of $\L$.
	Now, let $\phi\in C^\infty_{c}(\RR^d)$ be a smooth function with values in $[0,1]$ satisfying $\phi\equiv 1$ on $\L$, $\phi\equiv 0$ 
	on $\RR^d\setminus \{x\in\RR^d:\dist(x,\L)<L_0/2\}$, and $\norm{\grad\phi}_\infty^2\leq (2/L_0)^2$. 
	Then $\supp\phi \subset\tilde{\L}$, $\phi\psi\in H^1(\RR^d)$, and we obtain
	\eqs{
		\norm{V\psi}_{L^2(\L)}^2 &\leq \norm{V\phi\psi}_{L^2(\RR^d)}^2 \leq \lambda_1\norm{\grad(\phi\psi)}_{L^2(\RR^d)}^2 + \lambda_2\norm{\phi\psi}_{L^2(\RR^d)}^2\\
		&\leq 2\lambda_1(\norm{\grad\psi}_{L^2(\tilde{\L})}^2+\norm{\grad\phi}_\infty^2\norm{\psi}_{L^2(\tilde{\L})}^2)
		+ \lambda_2\norm{\psi}_{L^2(\tilde{\L})}^2 \\
		& \leq 2\cdot3^d\lambda_1 \norm{\grad\psi}_{L^2(\L)}^2 + 3^d(8\lambda_1/L_0^2+\lambda_2)\norm{\psi}_{L^2(\L)}^2.
		\qedhere
	}
\end{proof}

%
%
\section{Carleman estimates with an admissible potential}  \label{sec:Carleman}

In this section we present two Carleman estimates valid for Schr{\"o}dinger operators with admissible potentials. 
Both Carleman estimates improve or complement earlier results in the literature. 
The first Carleman estimate, cf.~Section~\ref{ssec:Carleman1}, goes back to~\cite{Vessella-03,EscauriazaV-03}, where an inequality of this kind 
is proven for a class of second order parabolic operators. 
In the elliptic setting, quantitative versions are proven for the pure Laplacian in~\cite{BourgainK-05,KleinT-16a}, and for second order elliptic 
operators in~\cite{NakicRT-19}. 
The second Carleman estimate, cf.~Section~\ref{ssec:Carleman2}, complements the Carleman estimate of~\cite{LebeauR-95} where 
second order elliptic operators are considered. 

Roughly speaking, our main observation is that we can add an admissible potential in an existing Carleman estimate. 
For our purposes, that is for the proof of Theorem~\ref{thm:ucp}, we implement this for the Carleman estimate given in~\cite{KleinT-16a} 
and a special case of the one in~\cite{LebeauR-95}. 

In the following, we denote by $\grad_{d+1}$ and $\Delta_{d+1}$ the gradient and the Laplacian on $\RR^{d+1}$, while 
$\grad$ and $\Delta$ denote the corresponding expressions on $\RR^d$.
By abusing notation slightly, for admissible $V\colon\RR^d\to\RR$ we use the same symbol to denote $V\colon\RR^d\times\RR\to\RR$ with $V(x,t) =V(x)$, $t\in\RR$.

\subsection{Generalization of the first Carleman estimate}  \label{ssec:Carleman1}

Let $\rho>0$, and define on $\RR^{d+1}$ the weight function $w$ by
\[
	w(y)=\phi(\lvert y \rvert/\rho),\quad \phi(r)=r\exp\lr -\int_0^r\frac{1-\e^{-t}}{t}\Diff{t}\rr.
\]
For future reference we note that $w$ satisfies 
\be
	\lvert y \rvert/(\rho\e)\leq w(y)\leq \lvert y \rvert/\rho\quad\text{and}\quad \lvert \grad w(y) \rvert^2 \leq w^2(y) / \lvert y \rvert^2 \leq 1 / \rho^2 
	\label{eq:Carleman-NRT-weight-est}
\ee
for all $y\in B_\rho(0)\setminus\{0\}$.

\begin{theorem} \label{thm:Carleman-NRT}
	Let $V \colon \RR^d \to \RR$ be admissible. 
	Then there are constants $\alpha_0,C_0\geq 1$ such that for all $\alpha\geq\alpha_0$ and all $\Psi\in H^2(\RR^{d+1})$ with support in $B_\rho(0)\setminus\{0\}$ we have
	\eq{
		\begin{split}
			\int_{\RR^{d+1}}\alpha\rho^2w^{1-2\alpha} \lvert\grad_{d+1}\Psi\rvert^2 &+ \alpha^3w^{-1-2\alpha}\lvert \Psi\rvert^2 \\
			&\leq C_0\rho^4\int_{\RR^{d+1}}w^{2-2\alpha}\lvert(-\Delta_{d+1}+V)\Psi\rvert^2.
		\end{split} 
		\label{eq:Carleman-est-NRT-Potential}
	}
	The dependency of $\alpha_0$ on $V$ and $\rho$ is quantified in~\eqref{eq:Carleman-NRT-const} below.
\end{theorem}

\begin{proof}
	The case $V\equiv 0$ in the theorem agrees with~\cite[Lemma~2.1]{KleinT-16a}.
	Let us denote the constants in this case by $\tilde{\alpha}_0,\tilde{C}_0\geq 1$. 
	It remains to show that we can insert $V$ on the right-hand side of~\eqref{eq:Carleman-est-NRT-Potential}. 
	To this end, we estimate $\lvert \Delta_{d+1}\Psi \rvert^2 \leq 2 \lvert(-\Delta_{d+1}+V)\Psi \rvert^2 + 2 \lvert V\Psi \rvert^2$ and subsume the resulting term $2\tilde{C}_0\rho^4I$ with 
	\[
		I = \int_{\RR^{d+1}}w^{2-2\alpha} \lvert V\Psi \rvert^2 = \int_\RR \norm{(Vw^{1-\alpha}\Psi)(\cdot,t)}^2_{L^2(\RR^d)}\Diff{t}
	\]
	in the left-hand side of~\eqref{eq:Carleman-est-NRT-Potential}
	by appropriate choices of $C_0$ and $\alpha_0$ that do not depend on $\Psi$. 
	More precisely, since $w$ is smooth on the support of $\Psi$, we have $w^{1-\alpha}(\cdot,t)\Psi(\cdot,t)\in H^1(\RR^d)$ for all $t\in\RR$.
	Thus
	\[
		I \leq \int_\RR \lambda_1\norm{\grad(w^{1-1\alpha}(\cdot,t)\Psi(\cdot,t))}_{L^2(\RR^d)}^2
		+\lambda_2\norm{w^{1-1\alpha}(\cdot,t)\Psi(\cdot,t)}_{L^2(\RR^d)}^2\Diff{t}.
	\]
	The product rule
	and~\eqref{eq:Carleman-NRT-weight-est} imply that the inequality
	\eq{
		\begin{split}
		\lvert \grad_{d+1}(w^{1-\alpha}\Psi) \rvert^2 
		&\leq 2 w^{2-2\alpha} \lvert \grad_{d+1}\Psi \rvert^2 + 2(\alpha-1)^2 \lvert \Psi \rvert^2 w^{-2\alpha}/\rho^2\\
		&\leq 2 w^{1-2\alpha} \lvert \grad_{d+1}\Psi \rvert^2 + 2(\alpha/\rho)^2w^{-1-2\alpha}\lvert \Psi \rvert^2
		\end{split} 
		\label{eq:Carleman-NRT-Gradient}
	}
	holds almost everywhere, where we have taken into account that both sides vanish outside of $B_\rho(0) \setminus \{0\}$. 
	Plugging this into the estimate for $I$,
	we see that we have proven~\eqref{eq:Carleman-est-NRT-Potential} with $C_0=4\tilde{C}_0$, provided that
	\bes
		\alpha\rho^2-4\lambda_1\tilde{C}_0\rho^4 \geq \alpha\rho^2/2\quad\text{and}\quad 
		\alpha^3-2\tilde{C}_0\rho^4(2\lambda_1(\alpha/\rho)^2+\lambda_2)\geq \alpha^3/2.
	\ees
	The latter is clearly satisfied for all $\alpha\geq \alpha_0$ with
	\be
		\alpha_0=\max\{\tilde{\alpha}_0,8\tilde{C}_0\lambda_1\rho^2+(4\tilde{C}_0\lambda_2\rho^4)^{1/3}\},
		\label{eq:Carleman-NRT-const}
	\ee
	which proves the claim.
\end{proof} 

\subsection{Generalization of the second Carleman estimate}\label{ssec:Carleman2}

We define the weight function
\be
	u\colon\RR^{d} \times \RR \ni (x,t)\mapsto -t+t^2/2- \lvert x \rvert^2/4\in\RR.
	\label{eq:Carleman-LR-weight}
\ee
For $\rho > 0$, let $B_\rho^+ =\{x\in\RR^{d+1}\colon \lvert x \rvert <\rho,\ x_{d+1}\geq 0\}$. 
Moreover, we denote by $C^\infty_{c,0}(B_\rho^+)$ 
the set of all functions $F\colon\RR^{d+1}_+\to\CC$, $\RR^{d+1}_+ =\RR^d\times [0,\infty)$, that satisfy $F(x,0)=0$ for all $x\in\RR^d$ and for which 
there exists a smooth function $\tilde{F}$ on $\RR^{d+1}$ with $\supp\tilde{F}\subset B_\rho$ satisfying $F=\tilde{F}$ on $\RR^{d+1}_+$.

\begin{theorem} \label{thm:Carleman-LR}
	Let $V\colon\RR^d \to\RR$ be admissible, let $\rho\in (0,2-\sqrt{2})$, and let $u$ be the weight function given in~\eqref{eq:Carleman-LR-weight}.
	Then there are constants $\beta_0,C_1\geq 1$ such that for all $\beta\geq\beta_0$ and all $\Psi\in C^\infty_{c,0}(B_\rho^+)$ we have
	\[
		\begin{split}
			\int_{\RR_+^{d+1}}&\e^{2\beta u}\lr \beta \lvert\grad_{d+1}\Psi \rvert^2+\beta^3\lvert \Psi \rvert^2\rr \\
			\leq
			C_1&\Bigg( \int_{\RR_+^{d+1}}\e^{2\beta u} \lvert (-\Delta_{d+1}+V)\Psi \rvert^2+\beta\int_{\RR^d}\e^{2\beta u(\cdot,0)} \lvert (\partial_t\Psi)(\cdot,0) \rvert^2\Bigg)
		\end{split}
	\]
	The dependency of $\beta_0$ on $V$ and $\rho$ is given in~\eqref{eq:Carleman-LR-const} below.
\end{theorem}

In the particular case where $V\equiv 0$, Theorem~\ref{thm:Carleman-LR} follows from the Carleman estimate given in Proposition~1 in the appendix of~\cite{LebeauR-95}. 
This Carleman estimate is formulated for arbitrary real-valued weight functions $u \in C^\infty (\RR^{d+1})$ satisfying 
\begin{enumerate}[(i)]
	\item $(\partial_{d+1} u ) (x) \not = 0$ for all $x\in B_\rho^+$, and
	\item for all $\xi \in \RR^{d+1}$ and $x \in B_\rho^+$ the implication 
	\bes 
		\left.
		\begin{array}{l}
		 2 \langle \xi , \nabla \psi \rangle = 0 \\[1ex]
		 \lvert \xi \rvert^2 = \lvert \nabla u \rvert^2
		\end{array}
		\right\}
		\quad \Rightarrow \quad
		\sum_{j,k=1}^{d+1} (\partial_{jk} u) \bigl(\xi_j \xi_k + (\partial_j u) (\partial_k u) \bigr) > 0 .
	\ees
	holds.
\end{enumerate} 
The particular weight function~\eqref{eq:Carleman-LR-weight} has been suggested in~\cite{JerisonL-99}. 
With this choice, (i) and (ii) are satisfied if $\rho \in (0,2-\sqrt{2})$.

\begin{proof}[Proof of Theorem~\ref{thm:Carleman-LR}]
	We have already noted that the theorem holds in the case that $V\equiv 0$. 
	Let $\tilde{\beta}_0,\tilde{C}_1\geq 1$ be the corresponding constants for this case. 
	The proof of the theorem is now analogous to the one of Theorem~\ref{thm:Carleman-NRT}. 
	We only need to replace~\eqref{eq:Carleman-NRT-Gradient} by
	\eqs{
		\lvert \grad(\e^{\beta u}\Psi) \rvert^2 
		&\leq 2\e^{2\beta u}\lr \lvert \grad\Psi \rvert^2
		+\beta^2\rho^2 \lvert \Psi \rvert^2/4\rr\quad\text{on}\quad \supp\Psi \subset B_\rho(0),
	}
	and choose  
	\be
		\beta_0 =\max\{\tilde{\beta}_0,2\tilde{C}_1\lambda_1\rho^2+(4\tilde{C}_1\lambda_2)^{1/3}\}.
		\label{eq:Carleman-LR-const}
	\ee
	This proves the theorem with $C_1=4\tilde{C}_1$.
\end{proof} 
%
%
\section{Unique continuation with singular potentials}\label{sec:UCP}
In this section we revisit the proof in~\cite{NakicTTV-18-JST}, which is in turn based on \cite{JerisonL-99}, and apply the
Carleman estimates obtained in the previous section to consider the class of admissible potentials $V$ in the context of unique continuation.
Note that in Section~\ref{sec:Carleman} we have assumed that the potential under consideration is admissible on the whole of $\RR^d$. 
In light of Lemma~\ref{lem:extension} and the preceding discussion, this is no restriction since we can always extend the potential $V$ on $\L$ in a suitable way.
Recall that the operator $H^\bullet_\L$ is self-adjoint and lower semibounded by Lemma~\ref{lem:operator-form-bnd}\,(a).

As in~\cite{JerisonL-99,NakicTTV-18-JST} we need to introduce the concept of ghost dimension in order to deal with spectral projections.
It allows to treat elements of a spectral subspace as eigenfunctions of a similar operator in higher dimensions. 

Let us abbreviate $H =H^\bullet_\L$, and denote by $(\mathcal{F}_t)_{t\in\RR}$ the family of unbounded self-adjoint operators
\[
	\mathcal{F}_t =\int_{-\infty}^\infty s_t(\lambda)\Diff{P_H(\lambda)},\quad
	s_t(\lambda) =\begin{cases} 
		\frac{\sinh(\sqrt{\lambda}t)}{\sqrt{\lambda}} & \lambda>0\\
		t & \lambda = 0 \\
		\frac{\sin(\sqrt{-\lambda}t)}{\sqrt{-\lambda}} & \lambda<0
	\end{cases},	
\]
in $L^2(\L)$, where $P_H(\lambda)$ is the spectral projection of $H$ associated with the interval $(-\infty,\lambda]$.
For fixed $\psi\in\Ran P_H(E)$ we define the function $\Psi \colon \L \times \RR \to \CC$ by
\be
	\Psi(\cdot,t) =\mathcal{F}_t\psi \in \Ran P_H(E) \subset \cD(H).
	\label{eq:extendedFunction}
\ee
This function belongs to $H^2(\Lambda\times (-T,T))$ for every $T>0$.
Moreover, it follows from~\cite[Lemma~2.5]{NakicTTV-18-JST} that $(\partial_t \Psi)(\cdot,0) = \psi$ and
\be
	H(\Psi(\cdot,t)) = (\partial_t^2\Psi)(\cdot,t).
	\label{eq:extendedOperator}
\ee

As in~\cite{NakicTTV-18-JST}, we can extend the operator $H$ to a sufficiently large generalized rectangle $\tilde{\L}$, where the
potential $V$ is extended as described in Section~\ref{sec:potentials}.
The accordingly extended functions $\psi$ and $\Psi$ inherit the properties mentioned above, in particular,~\eqref{eq:extendedOperator}. 
This enables us to apply the Carleman estimates from Section~\ref{sec:Carleman}.

For the sake of completeness, we recall in Appendix \ref{sec:ghost-dimension} certain aspects of the ghost dimension construction, which were left out in the last mentioned article.

The following lemma provides an appropriate replacement of~\cite[Proposition~2.9]{NakicTTV-18-JST} that allows us to come back from $\Psi$ to the original function $\psi$. 

\begin{lemma} \label{lem:norm-compare}
	Let $\tau>0$ and $\psi\in\Ran P_{H}(E)$, $E\geq 0$.
	Then
	\bes
		\frac{\tau}{2}\norm{\psi}_{L^2(\L)}^2 
		\leq \norm{\Psi}_{H^1(\L\times (-\tau,\tau))}^2
		\leq 2\tau(1+(1+\omega)\tau^2)e^{2\tau\sqrt{E}}\norm{\psi}_{L^2(\L)}^2,
	\ees
	where $\omega=(1+2\lambda_1E+\lambda_1^2+2\lambda_2)/2$.
\end{lemma}

\begin{proof}
	In the following we write $\Psi_t =\Psi(\cdot,t)\in\Ran P_H(E)$ for fixed $t$.
	
	For the proof of the lower bound we follow verbatim the proof of Proposition~2.9 in \cite{NakicTTV-18-JST}.
	We only need to adapt the reasoning for the upper bound to the present case of admissible $V$. 
	In view of~\eqref{eq:extendedOperator}, we obtain following \cite{NakicTTV-18-JST} that
	\[
		\int_\L \lvert \grad\Psi_t \rvert^2
		=-\int_\L V\lvert \Psi_t \rvert^2 + \int_\L \overline{\Psi}_t\partial_t^2\Psi_t
	\]
	and, therefore,
	\eqs{
		\norm{\Psi}_{H^1(\L\times (-\tau,\tau))}^2 
		&\leq \int_{-\tau}^\tau \int_\L(1+ \lvert V \rvert)\lvert\Psi_t \rvert^2+\lvert\overline{\Psi}_t \rvert \lvert (\partial_t^2\Psi_t) \rvert + \lvert \partial_t\Psi_t \rvert^2.
	}
	Using Young's inequality, we obtain for all $\mu > 0$ that
	\be
		\langle \Psi_t, \lvert V \rvert \Psi_t\rangle_{L^2(\L)}
		\leq
		\frac{1}{2\mu}\norm{V\Psi_t}_{L^2(\L)}^2 + \frac{\mu}{2}\norm{\Psi_t}_{L^2(\L)}^2
		\label{eq:comparison-of-norms-Young} .
	\ee
	Moreover, by Lemma~\ref{lem:operator-form-bnd}\,(b) we have 
	\eqs{
		\norm{V\Psi_t}_{L^2(\L)}^2
		&\leq
		\lambda_1\langle \Psi_t,-\Delta\Psi_t\rangle_{L^2(\L)} +\lambda_2\norm{\Psi_t}_{L^2(\L)}^2\\
		&\leq
		\lambda_1\left[\langle \Psi_t,H\Psi_t\rangle_{L^2(\L)} +\langle \Psi_t,\lvert V \rvert\Psi_t\rangle_{L^2(\L)}\right]
			+ \lambda_2\norm{\Psi_t}_{L^2(\L)}^2.
	}
	From this inequality and \eqref{eq:comparison-of-norms-Young} with $\mu = \lambda_1$ we find
	\be
		\begin{aligned}
			\norm{V\Psi_t}_{L^2(\L)}^2
			&\leq
			2\lambda_1\langle \Psi_t,H\Psi_t\rangle_{L^2(\L)}+2\left(\frac{\lambda_1^2}{2}+\lambda_2\right)\norm{\Psi_t}_{L^2(\L)}^2\\
			&\leq
			\left(2\lambda_1E+\lambda_1^2+2\lambda_2\right)\norm{\Psi_t}_{L^2(\L)}^2
			,
		\end{aligned}
		\label{eq:smoothing}
	\ee
	where for the last step we have taken into account that $\Psi_t\in\Ran P_H(E)$ and, thus, $\langle \Psi_t, H\Psi_t\rangle_{L^2(\L)} \leq E\norm{\Psi_t}_{L^2(\L)}^2$. 
	In turn, substituting~\eqref{eq:smoothing} into~\eqref{eq:comparison-of-norms-Young} with $\mu = 1$ gives
	\[
		\int_\L \lvert V \rvert \lvert \Psi_t \rvert^2
		\leq \frac{1}{2}(1+2\lambda_1E+\lambda_1^2+2\lambda_2)\norm{\Psi_t}_{L^2(\L)}^2
		=\omega \norm{\Psi_t}_{L^2(\L)}^2 .
	\]
	Thus, we obtain
	\[
		\norm{\Psi}_{H^1(\L\times (-\tau,\tau))}^2 
		\leq 2\int_{-\infty}^E I(\lambda)\Diff{\norm{P_{H}(\lambda)\psi}^2_{L^2(\L)}}
	\]
	with
	\eqs{
		I(\lambda)
		&=\int_0^\tau(\partial_ts_t(\lambda))^2+(\partial_t^2s_t(\lambda))s_t(\lambda)+(1+\omega)s_t(\lambda)^2\Diff{t}\\
		&\leq(1+\omega)\int_0^\tau (s_t(\lambda))^2\Diff{t}+(\partial_\tau s_\tau(\lambda))s_\tau(\lambda).
	}
	The rest of the proof is exactly as in~\cite{NakicTTV-18-JST}, but with $\norm{V}_\infty$ replaced by $\omega$.
\end{proof} 

\begin{remark}
	Inequality~\eqref{eq:smoothing} in the above proof formulates that the action of the admissible potential $V$ on elements of the spectral subspace for $H$ is
	comparable to that of a bounded potential. 
\end{remark} 

\begin{proof}[Proof of Theorem~\ref{thm:ucp}]
	It suffices to consider the case $G=1$; the general case then follows from a standard scaling argument, see the proof
	of~\cite[Corollary~2.2]{NakicTTV-18-JST}.
	
	We use $\cK_j$, $j\in\NN$, to denote constants depending only on the dimension $d$. 
	We first observe that with Lemma~\ref{lem:norm-compare} and the Carleman estimates from Section~\ref{sec:Carleman} at hand, we may follow the proof 
	of~\cite[Theorem~2.1]{NakicTTV-18-JST} for bounded potentials $V$; cf.~also the more comprehensive proof of~\cite[Theorem~3.17]{Taeufer-18}.
	Here one only needs to replace $\norm{V}_\infty$ by $\omega$ from the statement of Lemma~\ref{lem:norm-compare}.
	This leads to
	\[
		\norm{\psi}_{L^2(S_{\delta})} \geq \cK_1D_1^{-2} (D_2D_3)^{-2/\gamma}\norm{\psi}_{L^2(\L)},
	\]
	where $D_1,D_2,D_3$, and $\gamma$ satisfy
	\[
		D_1^2\leq\delta^{-\cK_2(1+\beta_0)},\
		D_2^{2/\gamma}\leq\delta^{-\cK_3(1+\alpha_0)}\ \text{and}\ D_3^{2/\gamma}\leq \delta^{-\cK_4(\log(1+\omega)+\sqrt{E})}.
	\]
	It remains to estimate $\alpha_0,\beta_0$ and $\log(1+\omega)$. 
	To this end, we observe 
	\[
		\log(1+\omega)
		\leq
		\cK_5\omega^{1/6}
		\leq
		\cK_6(1+\lambda_1+\lambda_2^{1/3}+\sqrt{E_+})
		,
	\]
	and recalling~\eqref{eq:Carleman-NRT-const},~\eqref{eq:Carleman-LR-const}, and since in the proof of~\cite[Theorem~2.1]{NakicTTV-18-JST} the parameter
	$\rho$ in the first Carleman estimate is bounded by a constant depending only on the dimension $d$, we have
	\[
		\alpha_0
		\leq
		\cK_7(1+\lambda_1+\lambda_2^{1/3})
		\quad\text{and}\quad
		\beta_0
		\leq
		\cK_8(1+\lambda_1+\lambda_2^{1/3}).
	\]
	This finishes the proof of the theorem. 
\end{proof}

\appendix 

\section{Remarks on ghost dimension} \label{sec:ghost-dimension}

We use the notation from Section~\ref{sec:UCP}.
Set $\kappa = \inf\sigma(H)$ and let $\psi \in \Ran P_H(E)$ for some $E \geq \kappa$.
According to~\cite[Lemma~2.5]{NakicTTV-18-JST}, the corresponding function $\Psi$ as defined in~\eqref{eq:extendedFunction} 
is infinitely $L^2(\Lambda)$-differentiable with respect to $t$ with derivatives
\be\label{eq:L2Diff}
	\partial_t^k \Psi(\cdot,t)
	=\biggl( \int_{[\kappa,E]} \partial_t^k s_t(\lambda) \Diff{P_H(\lambda) \biggr)\psi}\in\cD(H),\quad k \in \NN.
\ee
Here, $\partial_t^k \Psi(\cdot,t) \in \cD(H)$ follows from the fact that each $\partial_t^k s_t$ is bounded on $[\kappa,E]$.
Moreover, each $\partial_t^k \Psi$ belongs to $L^2(\Lambda \times (-T,T))$ for every $T > 0$.

\begin{lemma}
	$\Psi$ is infinitely weakly differentiable with respect to $t$, and the corresponding weak derivatives coincide with their $L^2(\Lambda)$ analogues.
\end{lemma}

\begin{proof}
	First we show that
	\[
		\lim_{h\to 0} \int_J \Bigl\lVert \partial_t^{k+1} \Psi(\cdot,t) - \frac{\partial_t^k \Psi(\cdot,t+h) -
		\partial_t^k \Psi(\cdot,t)}{h} \Bigr\rVert_{L^2(\Lambda)}^2 \Diff{t}
		=0
	\]
	for each bounded interval $J \subset \RR$ and each $k \in \NN_0$, where $\partial_t^k \Psi(\cdot,t)$ is given
	by~\eqref{eq:L2Diff}. To this end, it suffices to observe that
	\begin{align*}
		\Bigl\lVert\partial_t^{k+1} \Psi(\cdot,t) &- \frac{\partial_t^k \Psi(\cdot,t+h) - \partial_t^k \Psi(\cdot,t)}{h}\Bigr\rVert_{L^2(\Lambda)}^2\\
		&=\int_{[\kappa,E]}\Bigl\lvert\partial_t^{k+1} s_t(\lambda) - \frac{\partial_t^k s_{t+h}(\lambda) - 
		\partial_t^k s_t(\lambda)}{h}\Bigr\rvert^2\Diff{\langle P_H(\lambda)\psi , \psi \rangle}\\
		&\leq C h \norm{\psi}_{L^2(\Lambda)}^2
	\end{align*}
	with
	\[
		C = \sup_{(t,\lambda) \in \tilde{J} \times [\kappa,E]} \abs{ \partial_t^{k+2} s_t(\lambda) }^2 < \infty,
		\quad
		\tilde{J} = \{ t \pm \abs{h} \colon t \in J\},
	\]
	where we have taken into account the mean value theorem of differential calculus.

	Now, let $\varphi \in C_c^\infty(\Lambda \times \RR)$. The above then implies by Fubini's theorem and Cauchy-Schwarz inequality that
	\begin{multline*}
		\int_{\Lambda\times\RR} (\partial_t^{k+1}\Psi)(x,t) \varphi(x,t) \Diff{(x,t)}\\
		=\lim_{h\to 0} \int_{\Lambda\times\RR} \frac{(\partial_t^k \Psi)(x,t+h) - (\partial_t^k \Psi)(x,t)}{h}\varphi(x,t) \Diff{(x,t)}.
	\end{multline*}
	On the other hand, by change of variables with respect to $t$, we obtain
	\begin{align*}
		\int_{\Lambda\times\RR} &\frac{(\partial_t^k \Psi)(x,t+h) - (\partial_t^k \Psi)(x,t)}{h}\varphi(x,t) \Diff{(x,t)}\\
		&=\int_{\Lambda\times\RR} (\partial_t^k \Psi)(x,t) \frac{\varphi(x,t-h) - \varphi(x,t)}{h} \Diff{(x,t)}\\
		&\xrightarrow[h\to\infty]{}
		-\int_{\Lambda\times\RR} (\partial_t^k \Psi)(x,t) (\partial_t \varphi)(x,t) \Diff{(x,t)},
	\end{align*}
	where for the latter we have taken into account Lebesgue's dominated convergence theorem. The claim then follows by induction over $k$.
\end{proof}

\begin{lemma}
	For every $T > 0$ we have $\Psi \in H^2(\Lambda \times (-T,T))$.
\end{lemma}

\begin{proof}
	It follows from Lemma~\ref{lem:norm-compare} that $\Psi \in H^1(\Lambda \times (-T,T))$. Thus, it remains to show that
	\[
		\sum_{\substack{\abs{\mu}=2\\ \mu \in \NN_0^{d+1}}} \norm{\partial^\mu \Psi}_{L^2(\Lambda\times(-T,T))}^2<\infty.
	\]
	To this end, we write
	\be\label{eq:sndDer}
		\begin{aligned}
		\sum_{\substack{\abs{\mu}=2\\ \mu \in \NN_0^{d+1}}} \norm{\partial^\mu \Psi}&_{L^2(\Lambda\times(-T,T))}^2
		=\norm{ \partial_t^2 \Psi }_{L^2(\Lambda\times(-T,T))}^2\\
		&+\norm{ \nabla_x \partial_t \Psi }_{L^2(\Lambda\times(-T,T))}^2
		+\sum_{\substack{\abs{\mu}=2\\ \mu \in \NN_0^d}} \norm{\partial_x^\mu \Psi}_{L^2(\Lambda\times(-T,T))}^2.
		\end{aligned}
	\ee
	We already know that the first summand on the right-hand side is finite. In order to treat the second summand, we choose $c \geq 0$ with
	\[
		\langle (-\Delta_\Lambda^\bullet)f , f \rangle_{L^2(\Lambda)} \leq 2 \langle Hf , f \rangle_{L^2(\Lambda)} + c \norm{f}_{L^2(\Lambda)}^2
	\]
	for all $f \in \cD(\Delta_\Lambda^\bullet) = \cD(H)$, which is possible since $V$ is infinitesimally operator and thus also
	infinitesimally form bounded with respect to $-\Delta_\Lambda^\bullet$. Applying this for $f = (\partial_t \Psi)(\cdot,t)$ yields
	\begin{align*}
		\norm{ \nabla_x (\partial_t \Psi)(\cdot,t) }_{L^2(\Lambda)}^2
		&=\langle (-\Delta_\Lambda^\bullet) (\partial_t\Psi)(\cdot,t) , (\partial_t\Psi)(\cdot,t) \rangle_{L^2(\Lambda)}\\
		&\leq 2\langle H(\partial_t\Psi)(\cdot,t) , (\partial_t\Psi)(\cdot,t) \rangle_{L^2(\Lambda)} + c \norm{ (\partial_t\Psi)(\cdot,t) }_{L^2(\Lambda)}^2\\
		&=\int_{[\kappa,E]} (2\lambda + c) (\partial_t s_t(\lambda))^2 \Diff{\langle P_H(\lambda)\psi , \psi \rangle_{L^2(\Lambda)}}\\
		&\leq (2E+c)\norm{(\partial_t \Psi)(\cdot,t)}_{L^2(\Lambda)}^2.
	\end{align*}
	Integrating the latter over $t \in (-T,T)$ shows that also the second summand
	$\norm{ \nabla_x \partial_t \Psi }_{L^2(\Lambda\times(-T,T))}$ on the right-hand side of~\eqref{eq:sndDer} is finite.

	Let us now turn to the third summand. Since $V$ is infinitesimally operator bounded with respect to $\Delta_\Lambda^\bullet$, we may choose $\tilde{c} \geq 0$ with
	\[
		\norm{ \Delta_\Lambda^\bullet f }_{L^2(\Lambda)}^2
		\leq 4\norm{ Hf }_{L^2(\Lambda)}^2 + \tilde{c} \norm{ f }_{L^2(\Lambda)}^2
	\]
	for all $f \in \cD(\Delta_\Lambda^\bullet) = \cD(H)$. Using~\cite[Example~4.2]{Seelmann-21}, we now obtain for $f = \Psi(\cdot,t)$ that
	\begin{align*}
		\sum_{\substack{\abs{\mu}=2\\ \mu \in \NN_0^d}} \norm{\partial_x^\mu \Psi(\cdot,t)}_{L^2(\Lambda)}^2
		&\leq 2\sum_{\substack{\abs{\mu}=2\\ \mu \in \NN_0^d}} \frac{1}{\mu!}\norm{\partial_x^\mu \Psi(\cdot,t)}_{L^2(\Lambda)}^2\\
		&=2\norm{ \Delta_\Lambda^\bullet \Psi(\cdot,t) }_{L^2(\Lambda)}^2\\
		&\leq 8\norm{ H \Psi(\cdot,t) }_{L^2(\Lambda)}^2 + 2\tilde{c}\norm{ \Psi(\cdot,t) }_{L^2(\Lambda)}^2\\
		&=\int_{[\kappa,E]} (8\lambda^2 + 2\tilde{c}) s_t(\lambda)^2 \Diff{\langle P_H(\lambda)\psi , \psi \rangle_{L^2(\Lambda)}}\\
		&\leq(8\max\{ \abs{\kappa},\abs{E} \}^2 + 2\tilde{c}) \norm{\Psi(\cdot,t)}_{L^2(\Lambda)}^2.
	\end{align*}
	Integrating the latter over $t \in (-T,T)$ finally shows that also the third summand on the right-hand side
	of~\eqref{eq:sndDer} is finite. This completes the proof.
\end{proof}

\section*{Acknowledgments}

The authors are grateful to Ivan Veseli\'c for the suggestion to pursue this research direction and stimulating discussions during the completion of this work.

\newcommand{\etalchar}[1]{$^{#1}$}

\end{document}